\documentclass[oneside,leqno,12pt]{amsart}
\usepackage[english]{babel}
\usepackage[T1]{fontenc}
\usepackage{amsfonts}
\usepackage{amsmath}
\usepackage{amsthm}
\usepackage{dsfont}
\usepackage{amsthm}
\usepackage{amssymb}
\usepackage{amsxtra}     
\usepackage{epsfig}
\usepackage{verbatim}
\usepackage{color}

\newtheorem{theorem}{Theorem}
\newtheorem{prop}[theorem]{Proposition}
\newtheorem{corollary}[theorem]{Corollary}
\newtheorem{lemma}[theorem]{Lemma}

\newtheorem{definition}[theorem]{Definition}

\newtheorem{remark}[theorem]{Remark}
\numberwithin  {equation}{section}
\usepackage{color}

\newcommand{\R}{\mathbb{R}}
\usepackage{enumitem}

\begin{document}
\title{Positive solutions to sublinear elliptic problems}
\author{ Zeineb Ghardallou
\\ \\
\tiny{\emph{Faculty of Sciences of Tunis, Department of Mathematical Analysis and Applications,
University Tunis El Manar, LR11ES11, 2092 El Manar1, Tunis, Tunisia
\\Departement of Mathematics, University of Wroclaw, 50-384 Wroclaw, Poland}} }
\thanks{\emph{E-mail address}: zeineb.ghardallou@ipeit.rnu.tn/ zeineb.ghardallou@gmail.com}
\maketitle

\maketitle
\begin{abstract}
Let $L$ be a second order elliptic operator $L$ with smooth coefficients defined on a domain
$\Omega $ in $\mathbb{R}^d $, $d\geq3$, such that $L1\leq 0$. We study existence and properties
of continuous solutions to the following problem
 \begin{equation}\label{00}
Lu=\varphi(\cdot,u),
\end{equation}
in  $\Omega,$ where $\Omega$ is a Greenian domain for $L$ {(possibly unbounded)} in
$\mathbb{R}^d$ and $\varphi $ is a nonnegative function on $\Omega\times [0,+\infty [$ increasing
with respect to the second variable. By means of thinness, we obtain a characterization of
$\varphi$ for which \eqref{00} has a nonnegative nontrivial bounded solution.

\end{abstract}
\textbf{Keywords}. Nonlinear elliptic problems; Regular domain; Greenian domain; Thickness and
thinness at $\infty$.\\


\section{ Introduction}

{ Let $L$ be a second order elliptic operator with smooth coefficients defined on a domain $\Omega $. We assume that $L1\leq 0$ and $\Omega $ is Greenian for $L$. 
Let $\varphi:\Omega\times]-\infty ,\infty[\to [0,\infty[$ be a measurable function that satisfies
the following hypotheses:
\begin{description}
  \item[($H_1$)]$x\mapsto \varphi(x,c)\in \mathcal{K}_d^{loc}(\Omega)$ for every $c\in
      [0,+\infty[$. \footnote{See Preliminaries for the definition of a Greenian domain and the
      Kato class $\mathcal{K}_d$.}
  \item[$(H_2)$]$t\mapsto \varphi(x_0, t ) \hbox{  continuous increasing for every given } x_0\in
      \Omega.$

  \item[$(H_3)$]$\varphi(x,t)=0$ for every $x\in \Omega$ and $t\leq 0$.
\end{description}
We study positive continuous functions satisfying
 \begin{equation}\label{5}
     Lu=\varphi(\cdot,u),\  \hbox{in}\  \Omega;  \ {\mbox{in the sense of distributions}}.
\end{equation}
Our aim is to characterize solutions to \eqref{5} in general domains as thoroughly as possible,
preferably to obtain a one to one correspondence to $L$ harmonic functions, see Theorems
\ref{bigoneone} and \ref{further}. The only assumption we make on $\Omega $ is being Greenian, so
possibly unbounded, and  the general context is essential. For bounded regular domains the approach
is rather standard.

Our work has been inspired by the results of El Mabrouk \cite{Mabrouk} and El Mabrouk and Hansen
\cite{HM} who considered the equation
\begin{equation}\label{4}
\Delta u=p(x)u^{\alpha }, \hbox{ $0<\alpha <1$ and $p\in \mathcal{L}^{\infty}_{loc}(\mathbb{R}^d)$}.
\end{equation}
The main goal of the present paper is to show how methods of potential theory, applied so
effectively in \cite{Mabrouk, HM}, can be used to obtain results in much bigger generality.{ The
latter means not only the operator but, what more important, also the semi-linear part. In
particular,
 we improve considerably} the results known for the Laplace operator in $\R ^d$   \cite{LW}, \cite{L}.}

Equation \eqref{5} with $L$ being the Laplace operator on $\R ^d$ has recently attracted a lot of
attention \cite{AACH, Bandle, Mabrouk, HM, Guo, LS, LW, Lair, LM, Ahmed} with the semilinear part
$\varphi (x,u)$ being of the form $\varphi (x,u)=p(x)u ^{\alpha }+q(x)u^{\beta }$, $p,q $ positive
functions, $0<\alpha \leq \beta $ or more generally, $\varphi (x,u)=p(x)f(u)+q(x)g(u)$. Here we do
not assume that the variables in $\varphi$ are separated. { Neither we assume any more regularity
of $\varphi $ than provided by $(H_1)$ and $(H_2)$. Positivity and monotonicity of $\varphi$,
however, are crucial.


 Solution to \eqref{5} is $L$-subharmonic. As such it may be dominated by a { $L$-superharmonic function} or not i.e. it may be very large.  
In this paper we concentrate mainly on solutions of the first kind (dominated by an
$L$-superharmonic function) leaving the ``so called'' large solutions to a subsequent paper. Let
$G_{\Omega}$ be the Green function for $\Omega $. We prove in Section \ref{greeniandom}
that all the positive solutions $u$ in a Greenian domain $\Omega$ dominated by a positive
continuous $L$-superharmonic function $s$ are of the form
\begin{equation}\label{1} u(x)+\int_{\Omega } G_{\Omega }(x,y)\varphi(y,u(y))\,dy=h_u(x)\leq s(x),
\end{equation}
where $h_u$ is the minimal $L$-harmonic majorant of $u$ in $\Omega$.
In addition, the map $u\mapsto h_u$ defined for $u\leq s$ is injective. We also give sufficient
conditions implying surjectivity of it, see Theorem \ref{bigoneone}.
In particular, adding one more and a very natural hypothesis:

\medskip
$(H_4)$ For every $c\geq 0$, the Green potential $$G_{\Omega}(\varphi (\cdot , c))=\int_{\Omega}
G_{\Omega}(\cdot ,y)\varphi(y,c)\,dy$$ of $\varphi (\cdot , c)$ is finite at least on one
point.\footnote{Then it is continuous on $\Omega$, see Preliminaries.}

\medskip \noindent
we are able to establish one-to-one correspondence between bounded positive $L$-harmonic functions
and bounded positive solutions of \eqref{5} (Corollary \ref{oneonecoresp}).


In Section \ref{thinn}, we assume $(H_1)-(H_3)$, $L1=0$ and we formulate a necessary and sufficient
conditions for the existence of a nontrivial bounded solution to \eqref{5}, Theorem
\ref{sufficient-and-necessary-condition-existence-solution}. {It is boundedness  of
\begin{equation}\label{thincon}
\int _{\Omega \setminus A} G_{\Omega}(x,y )\varphi (y, c_0)\ dy \end{equation}
where 
the set $A$ is thin at $\infty$.}  \eqref{thincon} is essentially weaker than $(H_4)$ and it
generalizes considerably Theorem 3 in \cite{Mabrouk}.


At last, the author is grateful to her advisors Ewa Damek and Mohamed Sifi for their work, constant
encouragement and precious feedback. She also want to express their gratitude to {Krzysztof Bogdan,
Konrad Kolesko, and Mohamed Selmi} for their helpful and kindly suggestions.

\section{Preliminaries}
For every open set $\Omega$ of $\mathbb{R}^d$ with ($d\geq 3$) let $\mathcal{B}(\Omega) (resp. \
\mathcal{C}(\Omega))$ be the set of all real valued Borel measurable (resp. continuous) functions
on $\Omega$. We also consider $\mathcal{C}^1(\Omega)$ -  the space of continuously differentiable
functions on $\Omega$, { $\mathcal{C}^{\infty}(\Omega)$ - the space of infinitely differentiable
functions on $\Omega$, $\mathcal{C}^{\infty}_c(\Omega)$ - the space of functions in
$\mathcal{C}^{\infty}(\Omega)$ with compact support,} $\mathcal{C}^{2,\alpha}(\Omega)$ - the space
of functions with the second derivative being $\alpha$-H\"older continuous, $\mathcal{L}^1(\Omega)$
(resp. $\mathcal{L}^1_{loc}(\Omega)$)- the set of integrable (resp. locally integrable) functions
in $\Omega$, $\mathcal{L}^{\infty}(\Omega)$ (resp. $\mathcal{L}^{\infty}_{loc}(\Omega)$) the set of
bounded (resp. locally bounded) functions in $\Omega$. Also, for every set $\mathcal{F}$ of
numerical functions, we denote by $\mathcal{F}^+$ the set of all functions in $\mathcal{F}$ which
are nonnegative.

Let $\Omega$ be a domain in $ \mathbb{R}^d, d\geq 3$. { We assume that $L$ defined in $\Omega$ is a
second order elliptic operator with smooth coefficients i.e.
$$L=\sum_{1\leq i,j \leq d}a_{i,j}(x)\partial_i \partial_j+ \sum_{1\leq i \leq d} b_i(x)\partial_i+c(x) ,$$
 where  $a_{i,j}, b_i,c\in \mathcal{C}^{\infty }(\Omega )$, }$a_{i,j}(x)=a_{j,i}(x)$, $1\leq i,j\leq d$ and for every $x\in \Omega$ the quadratic form 
$$\sum_{1\leq i,j \leq d}a_{i,j}(x) \xi_i \xi_j  $$
is strictly positive definite. The latter means that $\displaystyle \sum_{1\leq i,j \leq d} a_{i,j}(x) \xi_i \xi_j >0$ for every $x\in \Omega$ and $\xi \in \mathbb{R}^d  \backslash   \{0\}.$ 
Notice that $L$ is locally uniformly elliptic in $\Omega$.

{ Suppose that $h\in \mathcal{L}^1_{loc}(\Omega )$ and $Lh=0$ in the sense of distributions. Then
$h\in\mathcal{C}^{\infty}(\Omega)$ \footnote{i.e. $h$ is equal a.e. to a smooth function} and
$Lh=0$ holds in the strong sense. Such functions will be called L-harmonic. { We denote
$\mathcal{H}(\Omega)$ the set of $L$-harmonic functions in $\Omega$.}
Let $v\in \mathcal{L}^1_{loc}(\Omega )$. We say that $v$ is $L$-subharmonic if $Lv\geq0$ in the
distributional sense. Then $v$ is equal a.e. to an upper semi-continuous function satisfying ``so
called'' sub-mean value property. A function $s$ such that $-s$ is $L$-subharmonic on $\Omega$ will
be called $L$-superharmonic on $\Omega$. We denote $\mathcal{S}(\Omega)$ the set of
$L$-superharmonic functions in $\Omega$.

Let $u\in \mathcal{C}^+(\Omega )$. We say that $u$ is a solution to equation \eqref{5} if
$\varphi(\cdot,u)$ is locally integrable on $\Omega $ and for all $\psi\in
\mathcal{C}_c^{\infty}(\Omega )$ we have
$$ \int_{\Omega } u L^*(\psi) - \int_{\Omega } \varphi(\cdot,u) \psi=0,$$ {where $L^*$ is the adjoint operator $L$ in $\Omega$.}
Notice that if $(H_1),(H_2)$ are satisfied, then $\varphi(\cdot,u)$ is always locally integrable on
$\Omega $.}

We say that $\Omega $ is Greenian if it has the Green function.\\
$G_{\Omega }: \Omega \times \Omega \to \R$ is called the Green function in $\Omega$ corresponding
to $L$ if $G_{\Omega }$ is $C^{\infty }$ outside $\{ (x,x): x\in \Omega \}$,
\begin{equation}\label{Greendelta}
LG_{\Omega}(\cdot,y)=-\delta_y, \hbox{ where $\delta_y$ denotes the Dirac measure at $y$,}\end{equation}
and if $0\leq h(x)\leq G(x,y)$, $Lh=0$ then $h=0$.

\begin{definition}
A Borel measurable function $\psi$ on $\Omega$ belongs locally to the Kato class (i.e. $\psi \in
\mathcal{K}_d^{loc}(\Omega)$) if $\psi$ satisfies \begin{equation}\label{kato-class}
\lim\limits_{\alpha\to 0}\sup_{x\in K}\int_{\Omega\cap(|x-y|\leq
\alpha)}\frac{|\psi(y)|}{|x-y|^{d-2}}\,dy=0
\end{equation}
for every compact set $K\subset \Omega $. If \eqref{kato-class} holds for every $x\in \Omega $
instead of $K$, we say that $\psi $ is in Kato class and we write $\psi \in \mathcal{K}_d(\Omega)$.

\end{definition}

Clearly, if $\psi\in \mathcal{L}^{\infty}_{loc}(\Omega)$ then $\psi\in
\mathcal{K}_d^{loc}(\Omega)$.
\begin{prop}(see e.g. ~\cite{Maagali/Malek})
Let $\psi\in \mathcal{K}_d(\Omega)$. Then for every $M>0$, we have
$$\int_{\Omega\cap(|y|\leq M)}|\psi(y)|\,dy<\infty.$$ In particular, if $\Omega$ is a bounded
domain, then $\psi\in \mathcal{L}^1(\Omega)$.
\end{prop}

Green potentials of functions belonging to the Kato class will play the main part in what follows.
We are going to recall their basic properties. For a more complete overview of potential theory we
refer the reader to the Appendix of \cite{Ghardallou} and Section 1 of \cite{Ghardallou-thesis}.

A bounded domain $D$ contained with its closure in $\Omega$ is called regular if each
$f\in\mathcal{C}(\partial{D})$ admits a continuous extension $H_Df$ on $\overline{D}$ such that
$H_Df$ is $L$-harmonic in $D$.
Let $G_D$ be the Green function for $L$ in $D$. Then by \cite{Miranda}, paragraph 8,
    there is $C>0$ such that
\begin{equation}\label{green}
G_D(x,y)\leq C|x-y|^{-d+2}, \hbox{  } D\times D.
\end{equation}
and so proceeding as in
 ~\cite{Maagali/Malek} we obtain
\begin{prop}\label{kato-class-in-bounded-domain}(see e.g. ~\cite{Maagali/Malek} and \cite{Hueber.Sieveking})

Let $D$ be a bounded regular domain in ${\mathbb{R}}^d$ ($d\geq 3$) and $\psi \in\mathcal{
K}_d(D)$, then
$$G_{D}\psi \in
\mathcal{C}_0(D).$$
\end{prop}

\begin{prop}(See \cite{Ghardallou})\label{contpot}
Let $\Omega$ be a Greenian domain, $\varphi\in \mathcal{K}_d^{loc}(\Omega)$ and there exists
$x_0\in \Omega$ such that $G_{\Omega}|\varphi |(x_0)$ is finite then $G_{\Omega}\varphi \in
\mathcal{C}(\Omega).$

\end{prop}


\section{Solution to $Lu=\varphi (\cdot , u)$ in a regular domain}\label{regular}

{\it Unless otherwise mentioned, throughout the paper we will suppose that $L$ is a second order
elliptic operator with smooth coefficients satisfying $L1\leq 0$ defined in the domain $\Omega $
that is Greenian for $L$ and $\varphi$ satisfies $(H_1)-(H_3)$.}

Let $D$ be a regular bounded domain such that $\overline{D}\subset\Omega$. Given $f\in C^+(\partial
D)$ there is $u\in C(\overline{ D})$ such that
 \begin{equation}
 \left\{
  \begin{array}{ll}
    Lu-\varphi(\cdot,u)=0, & \hbox{in $D$; in the sense of distributions;} \\
    u\geq 0,               & \hbox{in $D$;} \\
   u=f, & \hbox{on $\partial D$.}
  \end{array}
\right.
\label{3}
\end{equation}
Moreover, $u$ is related to the solution $H_Df$ of the classical Dirichlet problem with the
boundary data $f$ in the following way:
\begin{theorem}[Solution of \eqref{3}
 in a regular domain]\label{resolution-in-regular-domain}
Let $f\in\mathcal{C}^+(\partial D)$. Then there exists a unique solution $u$ 
to  \eqref{3}. Furthermore, we have:
 \begin{equation}\label{oneregul}
 u=H_Df(x)-\int_D G_D(x,y)\varphi(y,u(y))\,dy,\hbox{
  for every $ x \in D$}.\end{equation}
 \end{theorem}

 \begin{proof}
The statement was proved { for more general $\varphi$ } in \cite{BH} in the context of balayage
spaces. We may also refer the reader to \cite{Ghardallou} where the equation $Lu+\varphi (\cdot ,
u)=0$ was considered and the proof is completely analogous. However, for the readers convenience we
recall the definition of an operator, to which the Schauder theorem is applied.

 Given $f\in C^+(\partial D)$ let $$\beta=\sup_{x\in \overline{D}}H_Df(x) \hbox {   and   }
\alpha=\inf_{x\in D}[H_Df(x)- G_D(\varphi(\cdot,\beta))(x)].$$ We consider the set
$C=\{u\in\mathcal{C}(\overline{D}) , \alpha   \leq   u   \leq   \beta\} $ with the topology of
uniform convergence. So $C$ is a bounded, closed and convex.
 Let $T:\mathcal{C}({\overline{D}})     \to           \mathcal{C}({\overline{D}})$ be the map defined by
$$  u        \mapsto        H_Df-G_D(\varphi(\cdot,u)).$$
Since $u$ is bounded on $D$, $\varphi(\cdot,u)\in \mathcal{K}_d(D)$ 
and so  $G_D(\varphi(\cdot,u))\in \mathcal{C}_0(D)$. Therefore, $T$ is well defined and
$T(C)\subset C$. Indeed,$$\alpha\leq H_Df(x)-G_D(\varphi(\cdot,\beta))(x) \leq Tu(x)\leq H_Df(x)
\leq \beta,$$ for every $u\in C$ and every $x\in \overline{D}.$

Following the same lines of the proof of Theorem 8 in \cite{Ghardallou} and using the Schauder
Theorem we conclude that $T$ has a fixed point in $C$ i.e. $$u=H_Df-G_D(\varphi(\cdot,u)).$$

Moreover, 
$L(G_D(\varphi(\cdot,u)))=-\varphi(\cdot,u), $ and $ u=f \hbox { on }\partial D .$ By Lemma
\ref{comparaison-semi-elliptic} below, using the fact that zero is a trivial solution, we get that
$u$ is positive
 in $D$ and unique.
\end{proof}

 We complete the section with a Lemma that gives comparison between sub-solutions and super-solutions to \eqref{5} in $D$.
Suppose that $u$ is a continuous function. We say that $u$ is a subsolution if
 $Lu - \varphi (\cdot ,u)\geq 0$ or a supersolution if $Lu - \varphi (\cdot ,u)\leq 0$ in the sense of distributions.
   The Lemma holds in a considerable generality: we require only that the function $\varphi : D \times \R \mapsto \R $ is increasing with respect to the second variable.

\begin{lemma}[Comparison with values on the boundary]\label{comparaison-semi-elliptic}
\vspace{0.2cm}

Let $u,v \in \mathcal{C}(D)$, $Lu,Lv\in{\mathcal{L}}^1_{loc}(D)$  and let
$\varphi:D\times\mathbb{R}\rightarrow\mathbb{R}$ be an increasing function with respect to the
second variable . If $$\left\{
 \begin{array}{ll}
 Lu-\varphi(\cdot,u)\leq Lv-\varphi(\cdot,v), \\
 \liminf\limits_{\underset{y\in\partial D}{x\to y}}{(u-v)(x)}\geq0.
 \end{array}
 \right.
$$
Then:$$u-v\geq0\,\,in\,\,D. \footnote{ By the same argument, the result remains true for $D$
being unbounded domain under assumption $\liminf\limits_{\underset{y\in\partial D}{x\to
y}}{(u-v)(x)}\geq0 $ and $\liminf\limits_{|x|\to \infty} (u-v)(x)\geq0$. }$$ In particular, there
is a unique solution to \eqref{3}
\end{lemma}
\vspace{0.5cm}

\begin{proof}
Let $V=\{x\in D,\,\,u(x)<v(x)\}$, $V$ is open in $D$ because $u,v$ are continuous. As in
\cite{Ghardallou} we prove that
  $$\left\{
 \begin{array}{ll}
 L(u-v){\leq}0 &\hbox{ in $V$,}\\ 
 \liminf\limits_{x\to z}(u-v)(x)\geq 0 &\hbox{on $\partial V$.}
 \end{array}
 \right.
$$
It follows that $u-v$ is a lower semi-continuous function satisfying super-mean value property. In
the sense of the classical potential theory such functions are called $L$-superharmonic and they
satisfy a minimum principle that implies
 $$  u-v\geq 0 \hbox{ in $V$},$$
and so $V$ is empty. For the details we refer the reader to the Appendix of \cite{Ghardallou},
Proposition 42, and Section 1 of \cite{Ghardallou-thesis}.

\end{proof}

Proceeding as in \cite{Ghardallou,Ghardallou-thesis}, we obtain the following statement about
regularity of solutions.
\begin{theorem}\label{regularity-of-continuous-solution-in-bounded-domain}

Suppose that the assumptions of Theorem ~\ref{resolution-in-regular-domain} are satisfied and
additionally that for every $c>0$, $\varphi(\cdot,c)\in \mathcal{L}^{\infty}_{loc}(D), $ then the
unique solution $u$ of problem \eqref{3} belongs to $\mathcal{C}^+(\overline{ D})\cap
\mathcal{C}^1(D).$ Furthermore, if $\varphi$ $\in$ $\mathcal{C}^{\alpha}_{loc}(D\times [0,\infty[)$
then $u\in$ $\mathcal{C}^{2,\alpha}_{loc}(D)\cap\mathcal{C}(\bar{D}).$

\end{theorem}

\section{One-to-one correspondence in a Greenian domain}\label{greeniandom}
Let $\Omega $ be a Greenian domain for $L$. { We would like to obtain in $\Omega $ something in the
spirit of \eqref{oneregul}.} Clearly, in this case we cannot talk about boundary values but we may
write
\begin{equation}\label{oneingreen}
h=u+G_{\Omega}(\varphi(\cdot,u))\end{equation} and ask whether it gives a one-to-one correspondence
between positive solutions $u$ and $h\in \mathcal{H}^+(\Omega )$. In general, it is not the case
and there may exist solutions such that $G_{\Omega}(\varphi(\cdot,u))$ is not finite (see \cite{HM}
for the Laplace operator $\Delta $). However, under some more hypotheses we may get
\eqref{oneingreen} for $h$ from a suitable subset of $\mathcal{H}^+(\Omega )$,  
see Theorems \ref{bigoneone}, { \ref{further} and Corollary \ref{oneonecoresp}. The statements we
obtain are new even for $\Delta $.}

First we need to improve the comparison principle:

\begin{prop}\label{comparaison-of-solution}
Let $u_1,u_2\in\mathcal{C}^+(\Omega)$, $h_1,h_2\in\mathcal{C}^+(\Omega)$ such that:
 $$ h_i=u_i+G_{\Omega}(\varphi(\cdot,u_i)),\hbox{ $1\leq i \leq 2$}.$$
where $\varphi:\Omega\times[0,+\infty[\rightarrow[0,+\infty[$ satisfies $(H_1)-(H_2)$. If $h_1-h_2$
is positive $L$-superharmonic then
$$u_1-u_2\geq 0 \hbox{  in }\Omega.$$
{In particular, for every $h\in \mathcal{H}^+(\Omega)$ there exists at most one function
$u\in\mathcal{C}^+(\Omega)$ such that $$u+G_{\Omega}\varphi(\cdot,u)=h.$$}
\end{prop}

\begin{proof}
{We proceed as in the proof and the proof of Lemma 15 in \cite{Ghardallou}} applying the domination
principle (see e.g Proposition 44 in \cite{Ghardallou} or Proposition 35 in
\cite{Ghardallou-thesis}). Let
$$K=\{ x\in\Omega, (u_1-u_2)(x)\geq 0 \}.$$ By assumptions $K$ is closed and non empty. Let $$v=\varphi(\cdot,u_2)-\varphi(\cdot,u_1).$$ Then
$$ u_1-u_2 + G_{\Omega}(v^-)=h_1-h_2+G_{\Omega}(v^+),$$ with $t^+=\max\{t,0\}$ and $t^-=\max\{-t,0\}$. 
As in \cite{Ghardallou} $$L(G_{\Omega}(v^+))=-v^+, \hbox{ in }\Omega,$$ and
$$L(G_{\Omega}(v^-))=-v^-, \hbox{ in }\Omega.$$
Therefore, $h_1-h_2+G_{\Omega}(v^+)$ is a $L$-superharmonic positive function in $\Omega$ so it is
lower semi-continuous on $\overline{\Omega-K}$. In addition, $G_{\Omega}(v^-)$ is a potential
$L$-harmonic in $\Omega \setminus K$, because $v^-$ is supported in $K$.
Furthermore, it is clear that $G_{\Omega}(v^-)$ is continuous. 
Finally, 
$$h_1-h_2+G_{\Omega}(v^+)\geq G_{\Omega}(v^-),$$ on the boundary of $K$.

We can conclude
$$h_1-h_2+G_{\Omega}(v^+)\geq G_{\Omega}(v^-),$$ holds everywhere which implies that $u_1-u_2\geq0$
in $\Omega$.

\end{proof}

We will also need the following Lemma about convergence of a sequence of solutions:
\begin{lemma}\label{convergence-bounded-solution}

Let $\Omega$ be an open subset on $\mathbb{R}^d$, $\varphi$ satisfies $(H_1)-(H_2)$ and let $(u_n)$
be a sequence of nonnegative solutions of \eqref{5} in $\Omega$ that is uniformly bounded on
compact sets and it converges pointwise to a function $u$. Then $u$ is a solution of \eqref{5} in
$\Omega$.

\end{lemma}
\begin{proof}

Let $D$ a bounded regular domain such that $\bar{D} \subset \Omega$ and let $h_n$ be the
$L$-harmonic function in $D$ defined by $h_n=u_n+G_D(\varphi(\cdot,u_n))$.
 Since the functions $(u_n)$ restricted to $D$ are uniformly bounded on $D$, say by a constant $C_D$, we conclude by dominated convergence theorem that $G_D(\varphi(\cdot,u_n))$ converges to $G_D(\varphi(\cdot,u))$. This implies that the sequence $(h_n)$ is bounded above by $ C_D+G_D(\varphi(\cdot,C_D)) $ and it converges to a $L$-harmonic function $h$ such that $h=u+G_D(\varphi(\cdot,u))$.
 Hence, $u$ is continuous and satisfies \eqref{5} in $D$ and so in $\Omega$.

\end{proof}

{Finally, let  
$f\in\mathcal{C}^+(\Omega)$ and let $D$ be a regular bounded domain satisfying $\overline{D}\subset
\Omega$. For $x\in D$, we define
$$U_D^{\varphi}f(x) $$ to be the unique solution to the problem \eqref{3} and
$$U_D^{\varphi}f(x)=f(x), \hbox{ if $x\not\in D$}.$$ Then we have}

\begin{lemma}\label{properties-of-UD}\ \\
\begin{enumerate}[label=\Alph*]
  \item[$(a)$] Let $f,g\in \mathcal{C}^+(\Omega)$. Then:
 $U_{D }^{\varphi}$ is monotone nondecreasing i.e. : 
                 $$ U_{D }^{\varphi}f\leq U_{D }^{\varphi}g, \hbox{     if     } f\leq g \hbox{ in $\Omega$}.$$

  \item[$(b)$] Let $u\in\mathcal{C}^+(\Omega)$ a $L$-supersolution and
      $v\in\mathcal{C}^+(\Omega)$ a $L$-subsolution of \eqref{5} in $\Omega$. Suppose further
      that there is a subdomain $D'$ regular bounded such that $D'\subset D $. Then we have:
      \begin{itemize}
        \item[]($b_1$) $$   U_{D}^{\varphi}u\leq u \ \hbox{     and     }\ U_{D}^{\varphi}v\geq
            v.$$

        \item[]($b_2$) $$   U_{D '}^{\varphi}u \geq U_{D}^{\varphi}u\ \hbox{ and }\ U_{D
            '}^{\varphi}v\leq U_{D}^{\varphi}v.$$
      \end{itemize}

\end{enumerate}

\end{lemma}
The proof of Lemma \ref{properties-of-UD} is the same as in \cite{Mabrouk}.

\medskip
Now, let $(D_n)$ be a sequence of bounded regular domains such that for every $ n\in \mathbb{N},$
$\overline{D_n}\subset D_{n+1} \subset \Omega$ and $\displaystyle\bigcup _{n=1}^{\infty }D_n=\Omega
$. Such a sequence will be called {\it a regular exhaustion} of $\Omega $ and it is needed to
construct solutions to \eqref{5} in $\Omega $. The following proposition summarizes the basic
properties of this construction.

 \begin{prop}\label{existence of solution}
Let $s\in C^+(\Omega )$ be a $L$-superharmonic function. Then:
\begin{enumerate}[label=\Alph*]
  \item[$(i)$] The sequence $(U_{D_n}^{\varphi}s)$ is decreasing to a solution
      $u\in\mathcal{C}^+(\Omega)$ of \eqref{5} satisfying $u\leq s$.\footnote{Note here that $u$
      can be zero.}
\item[$(ii)$] Every solution $w\in C^+(\Omega )$ of \eqref{5} which is majorized  by $s$
    satisfies

    \begin{equation}\label{expression-solution}
    w + G_{\Omega}\varphi(\cdot,w)=h_w, \hbox{ in $\Omega$,}
\end{equation}
where $h_w$ is an $L$-harmonic minorant of $s$. Additionally, $w\leq u$ so $u$ is the maximal
one.
\item[$(iii)$] Let $h\in {\mathcal H}^+(\Omega )$ and $w\in C^+(\Omega )$ be such that
    \eqref{expression-solution} holds. Then $h_w$ is the smallest $L$-harmonic majorant of $w$
    and $w$ is the maximal solution to \eqref{5} which is majorized by $h_w$.
\end{enumerate}
\end{prop}

\begin{proof}\ \\
\begin{enumerate}[label=\Alph*]
  \item[$(i)$] In view of  Lemma \ref{properties-of-UD}: $$0 \leq U_{D_{n+1}}^{\varphi}s \leq
      U_{D_{n}}^{\varphi}s \leq s.  $$
So the limit $u=\lim\limits_{n\to +\infty} U_{D_n}^{\varphi}s$ exists by Lemma
      \ref{convergence-bounded-solution} and it gives a solution of \eqref{5} in $\Omega$.

 \item[$(ii)$] Let $w$ be a positive continuous solution of \eqref{5} in $\Omega$ bounded by $s$.
     Then by Lemma \ref{properties-of-UD} and Theorem \ref{resolution-in-regular-domain}
$$w = U_{D_n}^{\varphi}w \leq U_{D_n}^{\varphi}s .$$ When $n$ tends to $\infty$, we get $w\leq
u$. Also, by Theorem \ref{resolution-in-regular-domain}, $$ w + G_{D_n} \varphi(\cdot,w)=
H_{D_n}w, \hbox{ in $D_n$}.$$ Further, by the maximum principle
 $$0\leq H_{D_n}w \leq
H_{D_n}s \leq s, \hbox{ in $D_n$}.$$ Hence, the sequence $(H_{D_n}w)$ is increasing to a
$L$-harmonic function $h_w\leq s$ and by the monotone convergence theorem, $$ w + G_{\Omega}
\varphi(\cdot,w)= h_w, \hbox{ in $\Omega$}.$$

\item[$(iii)$] Let $h$ be an $L$-harmonic majorant of $w$. Then $H_{D_n}w \leq h$. When $n$ tends
    to infinity, we deduce $h_w \leq h$. If $w_1$ is a solution to \eqref{5} majorized by $h_w$
    and so by what has been said above
$$ w _1 +G_{\Omega} \varphi(\cdot,w_1)= h_{w_1}\leq h_w.$$ Hence by Proposition
\ref{comparaison-of-solution}, $w_1\leq w$.
\end{enumerate}

\end{proof}

Now we are ready to state a theorem about one-to-one correspondence between positive $L$-harmonic
functions and positive continuous solutions of \eqref{5}. Some further assumptions are needed to
guarantee that
$$u=\lim _{n\to \infty }U^{\varphi }_{D_n}s$$ constructed in Proposition \ref{existence of solution} is not trivial.
\begin{theorem}\label{bigoneone}
Let $s \in C^+(\Omega )$ be an $L$-superharmonic function. Assume that
$G_{\Omega}(\varphi(\cdot,s))$ is finite at least on one point.\footnote{By Proposition
\ref{contpot}, for every $x\in \Omega $, $G_{\Omega}(x,y)\varphi(y,s (y))$ is integrable as a
function of $y$.} Then
\begin{equation}\label{oneone2}
h=u+G_{\Omega}(\varphi(\cdot,u)) \hbox{ in }\Omega.
\end{equation}
gives one-to-one correspondence between $h \in \mathcal{H}^+(\Omega )$ bounded by $s $ and positive
continuous solutions of \eqref{5} bounded by $s $.
\end{theorem}

\begin{proof}

 Let $u\in C^+(\Omega )$ be a solution to \eqref{5} in $\Omega$ bounded by $s$. We denote $h_n=H_{D_n}u$. By Theorem \ref{resolution-in-regular-domain}
\begin{equation}\label{oneone4}
 H_{D_n}u=h_n=u+G_{D_n}(\varphi(\cdot,u)), \hbox{ in }D_n.\end{equation}

This implies that for every $n\in\mathbb{N}$,
\begin{equation}\label{oneone1}
h_{n+1}\geq h_n \hbox{  in }D_n.
\end{equation}
Indeed, $h_{n+1}\geq u=h_n \hbox{ on }\partial D_n$, hence \eqref{oneone1} follows by the maximum
principle.

Also,
$$ h_n(x)= u(x)\leq s(x) , \quad x\in \partial D$$
and so $h _n\leq s$ in $D_n$.

We deduce that $(h_n)$ converges to a $L$-harmonic function $h\leq s$. Moreover, by the monotone
convergence theorem applied to \eqref{oneone4}, $h$ satisfies
$$ h=u+G_{\Omega}(\varphi(\cdot,u)) \hbox{ in }\Omega.$$
Finally, Proposition \ref{comparaison-of-solution} implies that different solutions give rise to
different $L$-harmonic functions in \eqref{oneingreen}

Let now $h\leq s$ be a positive $L$-harmonic function in $\Omega$. By Proposition \ref{existence of
solution}, $(u_n)=(U_{D_n}^{\varphi}h)$ is decreasing to a solution
      $u\in\mathcal{C}^+(\Omega)$ of \eqref{5} satisfying $u\leq h$.

In addition \begin{equation}\label{corespn} h=u_n+G_{D_n}(\varphi(\cdot,u_n)), \hbox{ in }D_n.
\end{equation} and $$  0\leq G_{D_n}(x,y)\varphi(y,u_n(y))\leq G_{\Omega}(x,y)\varphi(y,s(y)).$$
But the right part is integrable in $\Omega$ as a function of $y$, so by dominated convergence
theorem
$$u+G_{\Omega}(\varphi(\cdot,u))=h, \hbox{  in }\Omega,$$  which implies that $u$ is not trivial, if
$h$ is not trivial too. Also by ($H_1$) and Proposition \ref{contpot}, $u$ is continuous in
$\Omega$.

\end{proof}



Notice that domination by $s$ is essential in the above argument. Given a $L$-harmonic function $h$
in $\Omega$, we can establish corresponding solutions in bounded subdomains $D_n$ but passing to
the limit we may end up with a trivial one unless \eqref{corespn} is preserved. The latter is
guaranteed by the fact that $G_{\Omega}(\varphi(\cdot,s))$ is finite. Finally, for bounded harmonic
functions and bounded solutions we have the following corollary.

\begin{corollary}\label{oneonecoresp}

Assume that $(H_4)$ is satisfied. Then
\begin{equation}\label{oneone}
h=u+G_{\Omega}(\varphi(\cdot,u)), \hbox{ in }\Omega.
\end{equation}
gives one-to-one correspondence between bounded functions in $\mathcal{H}^+(\Omega )$ and positive
continuous bounded solutions of \eqref{5}.

\end{corollary}

\begin{remark}
Notice that if bounded $L$-harmonic functions are known then \eqref{oneone} gives a description of
bounded solutions to \eqref{5}. See Section \ref{example} for an example.\end{remark}

{\begin{remark} \eqref{oneone} was proved by El Mabrouk for the operator $\Delta $ and $\varphi
(x,u)=\xi (x)u^{\gamma }$, $\xi \in L ^{\infty }_{loc}$, $0\leq \gamma \leq 1$ and, up to our
knowledge, there was no further development even for $\Delta $. \eqref{oneone} is also essentially
stronger  that Theorem 1 in \cite{L} where the equation $\Delta u-\xi (x)f(u)$ was considered for
$\xi\in C(\R ^d)$ and $f$ satisfying a so called ``Keller-Osserman'' condition.
\end{remark}}
A superharmonic function $s\in C^+(\Omega )$ in Theorem \ref{bigoneone} may be replaced by a
continuous solution to \eqref{5} such that $G_{\Omega}(\varphi(\cdot , v))$ is finite.
The statement is as follows:
\begin{theorem}\label{further} 
 Suppose that there is a continuous positive function $v$ and a positive $L$-harmonic function $h_v$ in $\Omega$
satisfying $$h_v=v+G_{\Omega}(\varphi(\cdot,v)).$$ Then there is one-to-one correspondence between
positive $L$-harmonic functions $h$ bounded by $h_v $ and positive continuous solutions of
\eqref{5} $u$ bounded {by $v $}, given by
\begin{equation}\label{oneone3}
h=u+G_{\Omega}(\varphi(\cdot,u)) \hbox{ in }\Omega.
\end{equation}

\end{theorem}

The above Theorem follows from {Lemmas \ref{integral-equation-one-one}} and \ref{ameliorate version
of dominated convergence theorem} below.

\begin{lemma}\label{integral-equation-one-one}
Suppose that there is a continuous positive function $v$ and a positive $L$-harmonic function $h_v$
in $\Omega$ satisfying $$h_v=v+G_{\Omega}(\varphi(\cdot,v)).$$ Then:
\begin{itemize}
\item $h_v= \lim\limits_{n \to +\infty} H_{D_n}v$ i.e. $h_v$ is the minimal $L$-harmonic function
    bounded below by $v$.
\item $v=\lim\limits_{n \to +\infty} U_{D_n}h_v$ i.e. $v$ is the maximal solution of \eqref{5}
    bounded above by $h_v$.
                                                                    \end{itemize}
\end{lemma}
\begin{proof}
First of all, clearly $v$ is a solution of \eqref{5} in $\Omega$ bounded by $h_v$.

Second, the sequence $(h_n)=(H_{D_n}v)$ is increasing to a positive $L$-harmonic function
$\tilde{h}$ such that $h_n\leq h_v$ and $$ v + G_{D_n}\varphi(\cdot,v)=h_n, \hbox{ in $D_n$}.$$ We
conclude by monotone convergence theorem
$$v+G_{\Omega}(\varphi(\cdot,v))=\tilde{h}, \hbox{ in }\Omega.$$ Consequently: $h_v=\tilde{h}$ in $\Omega$.

Third, it  remains to prove that $v$ is the maximal solution in the sense of Proposition
\ref{existence of solution}. Indeed, for the maximal one $u$ we have
$$
h_v= v+G_{\Omega}(\varphi(\cdot,v))\leq u+G_{\Omega}(\varphi(\cdot,u))\leq h_v.$$ Therefore, we
have equalities above and so $v=u$ in $\Omega$.
\end{proof}

\begin{lemma}\label{ameliorate version of dominated convergence theorem} Let $(f_n)$ and $(g_n)$ be
sequences of positive measurable functions on $\mathbb{R}^d$ such that \begin{itemize}
     \item $0\leq f_n \leq g_n$ in $\mathbb{R}^d$.
     \item $(f_n)$, $(g_n)$ converge pointwise respectively to $f$ and $g$ on  $\mathbb{R}^d$.
     \item For every $n\in \mathbb{N}$, $g_n$ and $g$ are
     integrable.
     \item $(\int_{ \mathbb{R}^d} g_n)$ converge to $\int_{ \mathbb{R}^d} g $.
   \end{itemize}
Then
\begin{equation}\label{fatou}
\lim_{n \to +\infty} \int_{ \mathbb{R}^d} f_n=\int_{ \mathbb{R}^d} f.
\end{equation}
\end{lemma}

\begin{proof}
Applying Fatou's Lemma to the sequences $(f_n)$ and $(g_n-f_n)$ we obtain $\displaystyle\lim_{n \to
+\infty} \int_{ \mathbb{R}^d} f_n=\int_{ \mathbb{R}^d} f$. In more details:
\begin{equation}\label{fatou1}
\int_{ \mathbb{R}^d} f= \int_{ \mathbb{R}^d} \liminf\limits_{n\to +\infty} f_n
\leq \liminf\limits_{n\to +\infty}\int_{ \mathbb{R}^d}  f_n .\end{equation}

Similary:
$$\int_{ \mathbb{R}^d} (g-f)= \int_{ \mathbb{R}^d} \liminf\limits_{n\to +\infty} (g_n-f_n)\leq
\liminf\limits_{n\to +\infty} \int_{ \mathbb{R}^d} (g_n-f_n) .$$

{But 
\begin{equation*}
  \begin{aligned}
     \liminf\limits_{n\to +\infty} \int_{ \mathbb{R}^d} (g_n-f_n) & = \int_{ \mathbb{R}^d} g+\liminf\limits_{n\to +\infty} (-\int_{ \mathbb{R}^d} f_n) \\
      & = \int_{ \mathbb{R}^d} g-\limsup\limits_{n\to +\infty} \int_{ \mathbb{R}^d} f_n
\end{aligned}
\end{equation*}
Hence
$$\int_{ \mathbb{R}^d} (g-f)\leq \int_{ \mathbb{R}^d} g-\limsup\limits_{n\to +\infty} \int_{ \mathbb{R}^d} f_n$$
and so
\begin{equation}\label{fatou2}
\int_{ \mathbb{R}^d} f\geq \limsup\limits_{n\to +\infty} \int_{ \mathbb{R}^d} f_n.\end{equation}
\eqref{fatou} follows now from \eqref{fatou1} and \eqref{fatou2}.}

\end{proof}





\begin{proof}[Proof of Theorem \ref{further}]
First of all, by Lemma \ref{integral-equation-one-one}, for $v_n= U_{D_n}^{\varphi}h_v$, we have
$v=\lim\limits _{n\to +\infty}v_n, v_n\leq u$ and so
$$\lim _{n\to +\infty} G_{D_n}(\varphi(\cdot,v_n))=G_{\Omega}(\varphi(\cdot,v)).$$

Let $u$ be a solution of \eqref{5} bounded by $v$. Then $(h_n)=(H_{D_n}u)$ is an increasing
sequence of $L$-harmonic functions bounded by $h_v$ satisfying
$$u+G_{D_n}(\varphi(\cdot,u))=h_n, \hbox{  in $D_n$},$$ such that
$\lim\limits_{n\to +\infty}h_n=h$. When $n$ tends to infinity, we get
$$h=u+G_{\Omega}(\varphi(\cdot,u)), \hbox{ in $\Omega$} .$$
On the other hand, let $h$ be a $L$-harmonic function bounded by $h_v$. Then $(u_n)=(U_{D_n}h)$ is
a decreasing sequence of solutions satisfying
$$u_n+G_{D_n}(\varphi(\cdot,u_n))=h, \hbox{  in $D_n$},$$ such that
$\lim\limits_{n\to +\infty}u_n=u$. Since $u_n|_{\partial D_n}=h\leq h_v=v_n |_{\partial D_n}$, by
Lemma \ref{comparaison-semi-elliptic}, $u_n\leq v_n$ in $D_n$.
Now applying lemma \ref{ameliorate version of dominated convergence theorem}, with
\begin{align*}
f_n(y)=&G_{D_n}(x,y)\varphi(y,u_n(y)) \mathds{1}_{D_n}\\
g_n(y)=&G_{D_n}(x,y)\varphi(y,v_n(y)) \mathds{1}_{D_n}\\
f(y)=&G_{\Omega}(x,y )\varphi(y,u(y))\mathds{1}_{\Omega}
\end{align*}
we conclude that  $\lim\limits_{n\to+\infty}
G_{D_n}(\varphi(\cdot,u_n))=G_{\Omega}(\varphi(\cdot,u))$ and so
$$h=u+G_{\Omega}(\varphi(\cdot,u)).$$
\end{proof}
\section{Operations on the nonlinear term of \eqref{5}}

{In this section we consider two functions $\varphi_1,\varphi_2:\Omega\times [0,\infty[\to
[0,\infty[$ satisfying $(H_1)-(H_3)$ and equation \eqref{5} related to $\varphi _1, \varphi _2$ and
$\varphi _1 + \varphi _2$.}

\begin{prop}\label{division-varphi}

Assume that for every $c>0$, $\varphi_1(\cdot,c)\leq \varphi_2(\cdot,c)$ and the equation \eqref{5}
has a nonnegative nontrivial solution for $\varphi=\varphi_2$ bounded {by an $L$-superharmonic
function $s$}. Then the same holds also for $\varphi=\varphi_1$.

\end{prop}

\begin{proof}

Let $w$ denotes a nontrivial solution of \eqref{5} for $\varphi=\varphi_2$ bounded by $s$. Then by
Proposition \ref{existence of solution}
$$w+G_{\Omega}(\varphi_2(\cdot,w))=h_w, \hbox{ in }\Omega,$$ where $h_w=\lim\limits _{n\to +\infty} H_{D_n}w$ and $(D_n)$ is
the regular exhaustion of $\Omega$. Let $u_n=U_{D_n}^{\varphi_1}h_w$ in $D_n$. It follows that
$u=\lim\limits_{n\to +\infty} u_n$ is a solution of \eqref{5} for $\varphi=\varphi_1$. Let us prove
that it is not trivial. By the fact $\varphi_1(\cdot,c)\leq \varphi_2(\cdot,c)$ for every $c>0$
$$ Lu_n-\varphi_2(\cdot,u_n)\leq Lu_n-\varphi_1(\cdot,u_n)= Lw-\varphi_1(\cdot,w)=0, \hbox{   in $D_n$.}$$

Also $u_n=h_w\geq w$ on $\partial D_n$. Then by Lemma \ref{comparaison-semi-elliptic}, we get
                         $$u_n\geq w, \hbox{   in $D_n$} .$$

Hence, $u=\lim\limits_{n\to +\infty} u_n$ is a nonnegative nontrivial solution of equation
\eqref{5} with $\varphi=\varphi_1$ bounded by $s$.

\end{proof}

\begin{prop}\label{divisionvarphi2}

Suppose, that there exists $h$, $h_1$ both $L$-harmonic functions such that $h_1\leq h$ and
$$h=u+G_{\Omega}(\varphi_1(\cdot,u)), \hbox{ in }\Omega,$$ where $u$ is a solution of the equation \eqref{5}
 for $\varphi=\varphi_1$.
 Assume further that $s=\lim\limits_{n\to +\infty} U_{D_n}^{\varphi_2}h_1$ is a nontrivial solution of \eqref{5} for $\varphi=\varphi_2$.
 Then $w=\lim\limits_{n\to +\infty} U_{D_n}^{\varphi_1+\varphi_2}h$ is a nontrivial solution of \eqref{5} for $\varphi=\varphi_1+\varphi_2.$

\end{prop}
\begin{proof}

At first, by Proposition \ref{existence of solution} $u=\lim\limits_{n\to +\infty}
U_{D_n}^{\varphi_1}h$. Secondly, $$h_1\leq h \Rightarrow U_{D_n}^{\varphi_2}h_1\leq
U_{D_n}^{\varphi_2}h,$$ then $v=\lim\limits_{n\to +\infty} U_{D_n}^{\varphi_2}h$ is a nontrivial
solution of \eqref{5} for $\varphi=\varphi_2$ bounded above by $h$ satisfying
$$v+G_{\Omega}(\varphi(\cdot,v))=h_v, \hbox{ in }\Omega,$$ where $h_v$ is the minimal $L$-harmonic
function bounded below by $v$.

 Thirdly, we denote $u_n=U_{D_n}^{\varphi_1}h,$$v_n=U_{D_n}^{\varphi_2}h$ and $w_n=U_{D_n}^{\varphi_1+\varphi_2}h$.

 As in the proof of Proposition \ref{division-varphi}, we get $u_n\geq w_n$ and $v_n\geq w_n$ in
 $D_n$. Let us prove that $w=\lim\limits_{n\to +\infty} w_n$ is non trivial.

 $$\left\{
 \begin{array}{ll}
 L(h+w_n-u_n-v_n)\leq 0, &\hbox{in $D_n$,}\\
 h+w_n-u_n-v_n=0, & \hbox{ on $\partial D_n$}.
 \end{array}
 \right.
$$
Hence in $D_n$  $$h+w_n-u_n-v_n\geq 0.$$

Then   $$(h-u_n)+(h_v-v_n)+w_n\geq h_v.$$ Tending $n$ to infinity, we get
$$G_{\Omega}(\varphi_1(\cdot,u))+G_{\Omega}(\varphi_2(\cdot,v))+w=(h-u)+(h_v-v)+w\geq h_v, \hbox{  in }\Omega.$$
Suppose now that $w=0$ then $$G_{\Omega}(\varphi_1(\cdot,u))+G_{\Omega}(\varphi_2(\cdot,v))\geq
h_v, \hbox{ in }\Omega .$$ Since $G_{\Omega}(\varphi_1(\cdot,u))+G_{\Omega}(\varphi_2(\cdot,v))$ is
a potential, we have $h_v=0$ implying $v=0$. The contradiction proves $w\neq 0$.

\end{proof}
\section{Existence of bounded solutions in Greenian domains}\label{thinn}
{\it Throughout this section,we assume that $L1=0$.}
Then $(H_4)$ is not necessary for existence of bounded solutions in a Greenian domain $\Omega $. A
sufficient and necessary condition for that is the integrability of
$$
G_{\Omega}(\cdot,y)\varphi(y,c_0)$$ on $\Omega \setminus A$, where $A$ is a ``so called'' thin set
at $\infty$, see Theorem \ref{sufficient-and-necessary-condition-existence-solution} below. {The
latter is proved below under fairly weak assumptions on $\varphi $ and it generalizes the results
known for $\Delta $ as well.} A subset $A\subset \Omega$ is called \emph{thin} if there is a
nonnegative continuous $L$-superharmonic function $s$ on $\Omega$ such that
 $$ s\geq 1 \hbox{ on } A\ \mbox {but}\ \inf _{x\in \Omega } s(x)< 1.$$
Given a particular operator it is often not so complicated to produce such a function and to obtain
existence of bounded solutions (see the example in the next section).

 We start
with the following theorem.

\begin{theorem}\label{sufficient-thin-at-infty}
If there exists $c>0$ such that $\{ \varphi(\cdot,c)>0\}$ is thin at $\infty$, then
$u_c=\lim\limits _{n\to +\infty} U_{D_n}^{\varphi}c$ is a nonnegative nontrivial bounded solution
of equation \eqref{5}.
\end{theorem}

\begin{proof}
We have to prove that $u_c$ is nontrivial. Let $B=\{\varphi(\cdot,c)>0\}$ and let $\tilde{s_0}$ be
a continuous nonnegative $L$-superharmonic function $\tilde{s_0}$ on $\Omega$ such that
$\tilde{s_0}\geq 1$ on $B$ and there is $x_0\in \Omega\backslash B$ such that $\tilde{s_0}(x_0)<1$.
We want to extend $B$ in order to obtain a closed thin set at $\infty$.
We denote $$A= \{x\in\Omega , \tilde{s}_0(x)\geq 1\} .$$ It is clear that $A$ is a closed subset of
$\Omega$ thin at $\infty$ containing $B$ i.e. $B\subset A$. Let $s=\inf\{c, cs_0\}$. $s$ is a
continuous $L$-superharmonic function such that $s(x_0)=cs_0(x_0)<c$. $cs_0\geq c$ on $A$, which
implies $s=c$ on $A$.

Let $D$ a bounded regular subdomain of $\Omega$ such that $\overline{D}\subset \Omega$ and $D\cap
A\not = \O$. We denote $u=U_D^{\varphi}c$ in $\Omega$. First we prove that
\begin{equation}\label{sdomin} s\geq G_D(\varphi(\cdot,u)), \hbox{  on $D$}.\end{equation}
\ref{sdomin} holds on $A\cap \overline {D}$ because $s\equiv c$ on  $A$ and $
G_D(\varphi(\cdot,u))\leq c$. Now we are going to use
the domination principle (Proposition 44 in \cite{Ghardallou} or Proposition 35 in
\cite{Ghardallou-thesis})
to prove that \ref{sdomin} holds in the rest of $D$ as well. Indeed, $s$ is a nonnegative
continuous $L$-superharmonic function in $\Omega$. In addition, $p= G_D(\varphi(\cdot,u))$ is a
continuous potential in $D$ satisfying $Lp=-\varphi(\cdot,u)$.
Also, $\Omega\backslash A\subset \Omega\backslash B,$ and $0\leq u\leq c$. Hence

$$\left\{
                                                                                       \begin{array}{ll}
                                                                                         0\leq \varphi(\cdot,u) \leq \varphi(\cdot,c), & \hbox{in $D$;} \\
                                                                                         \varphi(\cdot,c)=0, & \hbox{in $\Omega\backslash B$.}
                                                                                       \end{array}
                                                                                     \right.$$
meaning $p$ being $L$-harmonic in $D\backslash D\cap A.$ We are allowed then to apply the
domination principle and \eqref{sdomin} follows.

Now let $(D_n)$ be a regular exhaustion of $\Omega$, such that $D_0\cap A\not= \O$.
 Then, since constants are $L$-harmonic, by Theorem \ref{resolution-in-regular-domain}, for any $n\geq 1$  there is a positive continuous function $u_n$
in $D_n$ such that $$c=u_n+G_{D_n}(\varphi(\cdot,u_n)), \hbox{ in }D_n,$$ where
$u_n=U_{D_n}^{\varphi}c.$ As in \eqref{sdomin}, we have $$s\geq G_{D_n}(\varphi(\cdot,u_n))=c-u_n,
\hbox{ on $D_n$}.$$ When $n$ tends to $\infty$, we get $s\geq c-u_c$. In particular, $u(x_0)\geq
c-s(x_0)>0$ which implies that $u$ is nontrivial on $\Omega$.

\end{proof}


\begin{theorem}\label{sufficient-and-necessary-condition-existence-solution}\ \\


{Suppose that there exists a Borel set $A\subset \Omega$ which is thin at $\infty$ and $c_0>0$ such
that
\begin{equation}\label{thin}
\int_{\Omega\backslash A} G_{\Omega}(\cdot,y)\varphi(y,c_0)\,dy\not\equiv
\infty.\end{equation}
 Then equation \eqref{5} has a nonnegative nontrivial bounded solution in $\Omega$.
On the other hand, existence of a bounded solution to \eqref{5} implies that
 \begin{equation}\label{thin-bounded}
\int_{\Omega\backslash A} G_{\Omega}(\cdot,y)\varphi(y,c_0)\,dy
\end{equation} is bounded for a set $A$ which is thin at $\infty$.}
\end{theorem}

{\begin{remark} The above statement was proved by El Mabrouk for $\Delta $ and $\varphi (x,u)=\xi
(x)u^{\gamma }$, $\xi \in L ^{\infty }_{loc}$, $0\leq \gamma \leq 1$.
\end{remark}}


\begin{proof}
Suppose first that \eqref{thin} is satisfied. For every $c>0$, we have $$\varphi(\cdot,c)=1_A
\varphi(\cdot,c)+1_{\Omega-A} \varphi(\cdot,c)=\varphi_1(\cdot,c)+\varphi_2(\cdot,c).$$ Observe
that $\{ \varphi_1(\cdot,c)>0\}\subset A$ is thin at $\infty$, so in view of Theorem
\ref{sufficient-thin-at-infty}, $u_{c_0}=\lim\limits _{n\to +\infty } U_{D_n}^{\varphi_1}c_0$ is a
nonnegative nontrivial bounded solution to \eqref{5} with $\varphi=\varphi_1$. In addition
$$\int_{\Omega} G_{\Omega}(\cdot,y)\varphi_2(y,c_0)\,dy= \int_{\Omega\backslash A} G_{\Omega}(\cdot,y)\varphi(y,c_0)\,dy\not\equiv \infty.$$
Hence in view of Proposition \ref{existence of solution}, there exists a solution $v$ to \eqref{5}
with $\varphi=\varphi_2$ such that $$c_0=v+G_{\Omega}(\varphi_2(\cdot,v)), \hbox{  in  }\Omega.$$
Consequently, by Proposition \ref{divisionvarphi2}, $w=\lim\limits _{n\to +\infty}
U_{D_n}^{\varphi}c_0=\lim\limits _{n\to +\infty} U_{D_n}^{\varphi_1+\varphi_2}c_0$ is a nonnegative
nontrivial bounded solution  to \eqref{5}.

{Now, let $w$ be a nontrivial solution of \eqref{5} bounded by $c$.
Then $v=\lim\limits _{n \to +\infty }U_{D_n}^{\varphi}c$ is a nonnegative nontrivial solution of
equation \eqref{5} as well. Hence, there exist $0 < c_0 < c$ and $x_0,x_1\in\Omega$ such that $ 0<
v(x_0) \leq c_0 < v(x_1) \leq c$. Let $A=\{  v\leq {c_0} \}$ and $s=\frac{c-v}{ c - c_0 }$. Then
$s$ is a nonnegative
    continuous $L$-superharmonic function in $\Omega$, $s \geq 1$ on $A$ and $s(x_1)<1$. So
$A$ is thin at $\infty$ and
$$\int_{\Omega\backslash A} G_{\Omega}(\cdot,y)\varphi(y, c_0)\,dy  \leq  \int_{\Omega}
G_{\Omega}(\cdot,y)\varphi(y,v(y))\,dy\leq c.$$ Hence we get not only \eqref{thin} but also boundedness of
$\int_{\Omega\backslash A} G_{\Omega}(\cdot,y)\varphi(y,c_0)\,dy$.}
%
\end{proof}
\section{Example} \label{example}
Let $ \Omega =H_n = \{ (x,y), x\in\mathbb{R}^n , y>0   \}$ be the upper half space of
$\mathbb{R}^{n+1}$ and
$$\Delta =\sum_{k=1}^n \partial_k^2 + \partial_y^2$$ the Laplace operator in $\mathbb{R}^{n+1}$.
The corresponding Poisson kernel is given by ({ see \cite{Stein-Weiss}) $$ P(x,y) = c_n \frac{y}{
(|x|^2 + y^2)^\frac{n+1}{2} },$$ where $y\in \mathbb{R}_+^*$, $x\in \mathbb{R}^n$ and $c_n^{-1} =
\int_{\mathbb{R}^n} \frac{1}{( |x| + 1 )^{\frac{n+1}{2}}} \,dx$. Then every bounded
$\Delta$-harmonic function in $H_n$ is the Poisson integral of a bounded function in
$\mathbb{R}^n$. Then there is one-to-one correspondence between bounded positive $\Delta$-harmonic
function in $H_n$ and bounded positive functions in $\mathbb{R}^n$ given by
$$ h(x,y) = c_n \int_{\mathbb{R}^n} f(x-z)  \frac{y}{ (|z|^2 + y^2)^\frac{n+1}{2} }\,dz. $$ In
addition, $$\lim_{y\to 0} h(\cdot ,y) = f,\quad \mbox{weakly}$$ and $$ \sup_{x\in\mathbb{R}^n}
h(x,y) \leq ||f||_{\infty}, \hbox{ for every $y>0$.}$$

Let $\mathcal{L}^{\infty, +}(\mathbb{R}^n)$ be the set of bounded positive functions in
$\mathbb{R}^n$. It follows that if $\varphi$ satisfies $(H_1)-(H_4)$ then there is one-to-one
correspondence between $\mathcal{L}^{\infty, +}(\mathbb{R}^n)$ and continuous positive bounded
solutions of \eqref{5} in $H_n$ given by
$$f\to h \to u=h-G_{\Omega}\varphi(\cdot, u).$$
If $G_{\Omega}\varphi(\cdot, u)$ vanished at the boundary $\mathbb{R}^n\times \{0\}$ of $H_n$ then
the boundary values of $u$ and $h$ are the same.

It is easy to construct in this case thin sets and so to be able to apply criterion \eqref{thin}
for existence of bounded solutions. Let $s(x,y)=cy^{\gamma }$, for $0<\gamma <1$. Then $\Delta s
<0$, $s>1$ for $y>c^{-1\slash \gamma }$ and $s <1$ for $y<c^{-1\slash \gamma }$. This shows that
any set $A=\{ (x,y): y>c^{ -\frac{1}{\gamma}}\}$ is thin at $\infty$ and so for any function
$\varphi $ such that $\varphi (\cdot , c_0)$ is integrable against $G_{H_n}((0,1),\cdot )$ on $\{
(x,y): y<c^{ - \frac{1}{\gamma}}\}$ we have one-to-one correspondence of bounded solutions to
\eqref{5} and bounded $L$-harmonic functions.






\end{document}